\documentclass{amsart}
\usepackage{enumerate}
\usepackage{amssymb}
\usepackage{tikz}
\usepackage[normalem]{ulem}
\usepackage{xcolor}
\renewcommand{\subset}{\subseteq}

\renewcommand{\epsilon}{\varepsilon}

\newcommand{\st}{\colon\,}
\DeclareMathOperator{\Sat}{Sat}
\DeclareMathOperator{\sat}{sat}
\DeclareMathOperator{\satlim}{satlim}
\DeclareMathOperator{\wt}{wt}

\newtheorem{proposition}{Proposition}
\newtheorem{conjecture}[proposition]{Conjecture}
\newtheorem{theorem}[proposition]{Theorem}
\newtheorem{lemma}[proposition]{Lemma}

\newtheorem{corollary}[proposition]{Corollary}
\theoremstyle{definition}
\newtheorem{definition}[proposition]{Definition}
\theoremstyle{remark}

\newcommand{\sizeof}[1]{\left\lvert{#1}\right\rvert}

\newcommand{\reals}{\mathbb{R}}

\tikzstyle{vertex}=[inner sep = 0pt, minimum width=4pt, fill=black, shape=circle]
\newcommand{\gpoint}[2]{\node[style=vertex, label=#1:$#2$]}
\newcommand{\apoint}[1]{\gpoint{above}{#1}}

\title[A lower bound on the saturation number]{A lower bound on the saturation number, and graphs
  for which it is sharp}
\author{Alex Cameron \and Gregory J.~Puleo}
\begin{document}
\begin{abstract}
  Let $H$ be a fixed graph. We say that a graph $G$ is $H$-saturated
  if it has no subgraph isomorphic to $H$, but the addition of any
  edge to $G$ results in an $H$-subgraph.  The \emph{saturation
    number} $\sat(H,n)$ is the minimum number of edges in an
  $H$-saturated graph on $n$ vertices. K\'{a}szonyi and Tuza, in 1986,
  gave a general upper bound on the saturation number of a graph $H$,
  but a nontrivial lower bound has remained elusive. In this paper we
  give a general lower bound on $\sat(H,n)$ and prove that it is
  asymptotically sharp (up to an additive constant) on a large class
  of graphs.  This class includes all threshold graphs and many graphs
  for which the saturation number was previously determined
  exactly. Our work thus gives an asymptotic common generalization of
  several earlier results. The class also includes disjoint unions of
  cliques, allowing us to address an open problem of Faudree, Ferrara,
  Gould, and Jacobson.
\end{abstract}
\maketitle
\section{Introduction}

Given a fixed forbidden graph $H$, what is the minimum number of edges
that any graph $G$ on $n$ vertices can have such that $G$ contains no
copy of $H$, but the addition of any single edge to $G$ results in a
copy of $H$? This question is a variation of the well-known forbidden
subgraph problem in extremal graph theory, which asks for the maximum
number of edges in an $H$-free graph on $n$ vertices. Asking for
  the \emph{minimum} number of edges instead (and tailoring the
  definition so that this is a sensible question) yields the notion of
  the \emph{saturation number} of a graph $H$, first defined by
  Erd\H{o}s, Hajnal, and Moon~\cite{EHM}, albeit with slightly
  different terminology.

  \begin{definition}
    Let $H$ be a graph. For any graph $G$, we say that $G$ is
    \emph{$H$-free} if it contains no subgraph isomorphic to $H$. We
    say that $G$ is \emph{$H$-saturated} if it is $H$-free and for any
    $xy \in \overline{E(G)}$, the graph $G+xy$ contains a subgraph
    isomorphic to $H$. For $n \geq \sizeof{V(H)}$, let
    $\Sat(H,n)$ denote the set of all $H$-saturated graphs on $n$
    vertices, and let the \emph{saturation number} of $H$
    be \[\sat(H,n) = \min_{G \in \Sat(H,n)} {|E(G)|}.\] In the event
    that $\Sat(H,n) = \emptyset$, we adopt the convention that
    $\sat(H,n)=\infty$. Note that this will only happen if $H$ has no
    edges.
  \end{definition}

  In their paper introducing the concept, Erd\H{o}s, Hajnal, and
  Moon~\cite{EHM} determined $\sat(H,n)$ in the case where $H$ is a
  complete graph.  Since then, the saturation numbers have been
  determined for various classes of graphs. A nice survey on these
  results and more was written by Faudree, Faudree, and Schmitt
  \cite{survey}.

The best known general upper bound on $\sat(H,n)$ was given by
K\'{a}szonyi and Tuza \cite{KT} and later slightly improved by Faudree
and Gould~\cite{FG13}. However, as mentioned in \cite{survey} and in
\cite{FFGJ09}, there is no known nontrivial general lower bound on
this function. In this paper we give such a bound, and determine a
class of graphs for which this bound is asymptotically sharp: for such
graphs, we can prove that $\sat(H,n) = \alpha_H n + O(1)$, where
$\alpha_H$ and the $O(1)$ term depend on only $H$. (We remark that it
is not known, in general, that the limit
$\lim_{n \to \infty}\frac{\sat(H,n)}{n}$ even exists, even though it is known~\cite{KT} that $\sat(H,n)$ is always
bounded by a linear function of $n$; the existence of
this limit was stated as a conjecture by Tuza~\cite{Tuza88}.)

This class of graphs includes all threshold graphs as well as some
non-threshold graphs. In particular, many previously-studied classes
of graphs fall into this class, including cliques~\cite{EHM},
stars~\cite{KT}, generalized books~\cite{CFG08}, disjoint unions of
cliques~\cite{FFGJ09}, generalized friendship graphs~\cite{FFGJ09},
and several of the ``nearly complete'' graphs of~\cite{FG13}. Our
result can be considered as an asymptotic common generalization of
these previous results: at the cost of no longer determining the
\emph{exact} saturation number as in the previous results, we obtain a
simple unified proof that gives the saturation number up to an
additive constant number of edges.

The rest of the paper is organized as follows. In
Section~\ref{sec:weight} we state and prove our general lower bound on
the saturation number, and prove an upper bound on the saturation
number of the graph $H'$ obtained from a graph $H$ by adding a
dominating vertex. In Section~\ref{sec:sat-sharp}, we define the
\emph{sat-sharp} graphs to be the graphs whose saturation numbers are
asymptotically equal to the lower bound of Section~\ref{sec:weight},
and prove that this class of graphs is closed under adding
isolated vertices and dominating vertices. In
Section~\ref{sec:threshold} we discuss threshold graphs, which are
contained within the class of sat-sharp graphs and encompass several graphs
whose saturation numbers were previously determined. Finally, in
Section~\ref{sec:disjoint-cliques}, we prove that any graph consisting
of a disjoint union of cliques is sat-sharp, and discuss the
implications of this.

\section{A weight function and some general bounds}\label{sec:weight}
In this section, we will define a weight function for a general graph
$H$, and prove that it gives a lower bound on the saturation number
$\sat(H,n)$. We will also prove a general bound relating the
saturation number of $H$ to the saturation number of the graph
$H'$ obtained from $H$ by adding a dominating vertex.
\begin{definition}
  For a vertex $x$ in a graph $G$, let $N_G(x)$ and $N_G[x]$ denote the
  \emph{open} and \emph{closed neighborhoods} of $x$ respectively:
  \begin{align*}
    N_G(x)&=\{y \in V(G)\st xy \in E(G)\}, \\
    N_G[x]&=N_G(x) \cup \{x\}.
  \end{align*}
  Let $d_G(x) = \sizeof{N_G(x)}$ denote the degree of $x$,
  and for a vertex set $S$, let $d_{G,S}(x)$ denote the number of
  neighbors of $x$ in the set $S$:
  \[d_{G,S}(x) = \sizeof{N(x) \cap S}.\]
  When the graph $G$ is clear from context, we omit it from our
    notation and simply write $N(v)$, $d(v)$, or $d_S(v)$ as
    appropriate.
\end{definition}
\begin{definition}
  Let $uv$ be an edge in a graph with $d(u) \leq d(v)$. Define the
  \emph{weight} $\wt(uv)$ of the edge $uv$ by
  \[ \wt(uv) = 2\sizeof{N(u) \cap N(v)} + \sizeof{N(v) - N(u)}. \]
  Define the weight of the graph $H$ by
  \[ \wt(H) = \min_{uv \in E(H)} \wt(uv). \]
  If $E(H)=\emptyset$, we define $\wt(H)=\infty$.
\end{definition}

\begin{lemma}\label{lem:wtlower}
  For every graph $H$, there exists a constant $c'_H$ such that
  \[ \sat(H, n) \geq \frac{\wt(H)-1}{2}n - c'_H. \]
\end{lemma}
\begin{proof}
  First observe $\wt(H) \geq 1$ for all $H$ and that the claim is
    trivial when $\wt(H) = 1$, so (as $\wt(H)$ is an integer) we may
    assume that $\wt(H) \geq 2$.  Let $G$ be an $H$-saturated graph,
  let $x^*$ be a vertex of minimum degree in $G$, and let
  $B = N_G(x^*)$, so that $\sizeof{B} = d_G(x^*)$. Observe that if
  $d_G(x^*) \geq \wt(H)-1$, then the degree-sum formula immediately
  gives $\sizeof{E(G)} \geq \frac{\wt(H)-1}{2}n$, so we may assume
  that $d_G(x^*) < \wt(H) - 1$. As both of these quantities are
  integers, we have $d_G(x^*) \leq \wt(H) - 2$.

  Consider any vertex $y \in V(G) - N[x^*]$.  By hypothesis, the graph
  $G+x^*y$ contains a copy of $H$.  Let $\phi : V(H) \to V(G+x^*y)$ be
  an embedding of $H$ into $G+x^*y$. Since $G$ is $H$-free, the new
  edge $x^*y$ must be the image of some edge $uv \in E(H)$. We may
  take our notation so that $d_H(u) \leq d_H(v)$.  Let
  $b = \sizeof{N_H(u) \cap N_H(v)}$ and let $a = d_H(v) - b$,
  so that $\wt(uv) = a + 2b$.

  We first claim that $d_G(y) \geq a+b-1$; this requires considering
  two cases, depending on whether $y = \phi(u)$ or $y = \phi(v)$.
  If $y = \phi(u)$, then
  \[d_G(y) \geq \delta(G) = d_G(x^*) \geq d_H(v) - 1 = a+b-1.\]
  Similarly, if $y = \phi(v)$, then
  $d_G(y) \geq d_H(v) -1 \geq a+b-1$. This establishes the claim.

  Now, observe that regardless of whether $y = \phi(u)$ or
  $y = \phi(v)$, we have
  \[\phi(N_H(u) \cap N_H(v)) \subset N_G(x^*) = B.\] So in
  $G$, the vertex $y$ has $b$ guaranteed neighbors in $B$, together
  with at least $a-1$ additional edges which may go to $B$ or may go
  to $\bar{B} - x^*$, where $\bar{B} = V(G) - B$. Therefore,
  \[2d_{G,B}(y) + d_{G,\bar{B}}(y) \geq 2b + a - 1 = \wt(uv)-1 \geq \wt(H)-1\]
  for all $y \in \bar{B} - x^*$.
  
  Now, note that $\sum_{x \in B}d_G(x) \geq \sum_{y \in \bar{B}} d_{G,B}(y)$. So it follows that
  \begin{align*}
    \sizeof{E(G)} &= \frac{1}{2}\left(\sum_{x \in B} d_G(x)+\sum_{y \in \bar{B}} d_G(y) \right)\\
                  &\geq \frac{1}{2}\left(\sum_{y \in \bar{B}} d_{G,B}(y) +\sum_{y \in \bar{B}} d_{G,B}(y) \right)\\
                  &= \frac{1}{2}\sum_{y \in \bar{B}} \left(2d_{G,B}(y)+d_{G,\bar{B}}(y)\right)\\
                  &\geq \frac{2d_G(x^*)+\left(\wt(H)-1\right)\sizeof{\bar{B} - x^*}}{2}\\    
                  &= \frac{2d_G(x^*)+\left(\wt(H)-1\right)\left(n-1-d_G(x^*)\right)}{2}\\
                  &= \frac{\wt(H)-1}{2}n - c'_H,
  \end{align*}
  where \[c'_H = \frac{d_G(x^*)(\wt(H) - 3)+(\wt(H) - 1)}{2}.\] Since we
  assume $0 \leq d_G(x^*) \leq \wt(H)-2$, the value $c'_H$, considered
  as a formal function of the quantity $d_G(x^*)$, is
  maximized at $d_G(x^*) = \wt(H)-2$ whenever $\wt(H) \geq 2$ (the
  case $\wt(H) = 2$, which would imply that this formal function has a negative derivative
    in $d_G(x^*)$, implies that $d_G(x^*)=0$). Therefore,
  \[ \sat(H, n) \geq \frac{\wt(H)-1}{2}n -
    \frac{\wt(H)^2-4\wt(H)+5}{2} \] for $\wt(H) \geq 2$.
\end{proof}
\begin{figure}
  \centering
  \begin{tikzpicture}
    \begin{scope}[xshift=-2cm]
      \apoint{} (u) at (-.67cm, 0cm) {};
      \apoint{} (v) at (.67cm, 0cm) {};
      \apoint{} (w) at (0cm, 1cm) {};
      \apoint{} (z1) at (-.5cm, 1.5cm) {};
      \apoint{} (z2) at (.5cm, 1.5cm) {};
      \draw (w) -- (v) -- (u) -- (w) -- (z1);
      \draw (w) -- (z2);
      \node[anchor=north] at (0cm, -1em) {$H$};
    \end{scope}
    \begin{scope}[xshift=2cm]
      \apoint{} (u) at (-.67cm, 0cm) {};
      \apoint{} (v) at (.67cm, 0cm) {};
      \apoint{} (w) at (0cm, 1cm) {};
      \apoint{} (z1) at (-.5cm, 1.5cm) {};
      \apoint{} (z2) at (.5cm, 1.5cm) {};
      \apoint{v^*} (x) at (0cm, 2cm) {};
      \draw (w) -- (v) -- (u) -- (w) -- (z1);
      \draw (w) -- (z2);
      \draw (x) -- (z1); \draw (x) -- (z2); \draw (x) -- (w);
      \draw (x) .. controls ++(0:1cm) and ++(45:1cm) .. (v);
      \draw (x) .. controls ++(180:1cm) and ++(135:1cm) .. (u);      
      \node[anchor=north] at (0cm, -1em) {$H'$};
    \end{scope}    
  \end{tikzpicture}
  \caption{Forming $H'$ by adding a dominating vertex $v^*$ to the
    graph $H$.}
  \label{fig:add-dom}
\end{figure}
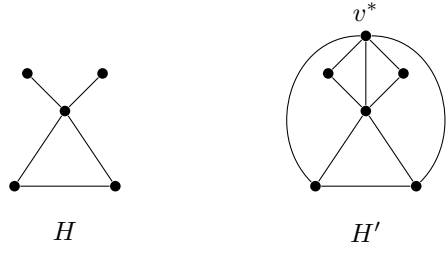
A central goal of this paper is to explore the effect on the
saturation number of the operation of adding a dominating vertex to
$H$, as shown in Figure~\ref{fig:add-dom}. It turns out that
  this gives a general upper bound on the saturation number of the new
  graph in terms of the saturation number of $H$; we wish to know when
  this upper bound is sharp. We believe that this upper bound is in
  the same general spirit as Lemma 9 of K\'aszonyi--Tuza~\cite{KT}.
\begin{lemma}\label{lem:adddom}
  If $H'$ is obtained from $H$ by adding a dominating vertex $v^*$,
  then for all $n \geq \sizeof{V(H')}$, we have
  $\sat(H', n) \leq (n-1) + \sat(H, n-1)$.
\end{lemma}
\begin{proof}
  It suffices to produce an $H'$-saturated graph with at most the
  indicated number of edges.  Let $G$ be a minimum $H$-saturated graph
  on $n-1$ vertices, and let $G'$ be obtained from $G$ by adding a new
  dominating vertex $x^*$. It is clear that
  $\sizeof{E(G')} = (n-1) + \sat(H, n-1)$; we show that $G'$ is
  $H'$-saturated.

  First we argue that $G'$ is $H'$-free. Suppose to the contrary that
  $\phi : V(H') \to V(G')$ is an embedding of $H'$ into $G'$. If
  $x^* \notin \phi(V(H'))$, then $\phi$ is an embedding of $H'$ into
  $V(G')$.  Hence $G'$ has a copy of $H'$ and thus a copy of $H$,
  contradicting the $H$-saturation of $G$.  If $\phi(v^*) = x^*$, then
  the restriction of $\phi$ to $V(H)$ is an embedding of $H$ into $G$,
  again a contradiction.

  Hence we can assume that $\phi(v^*) \neq x^*$ and there is some
  vertex $w^* \in V(H)$ with $\phi(w^*) = x^*$. Construct a new
  embedding $\phi_0 : V(H) \to V(G)$ by letting
  $\phi_0(w^*) = \phi(v^*)$ and taking $\phi_0(w) = \phi(w)$ for all
  $w \neq w^*$. Since $\phi(v^*)$ dominates the image of $\phi$ in $G$
  (as $x^*$ was a dominating vertex of $H$), we see that $\phi_0$ is
  still a valid embedding. Hence we have again obtained a copy of $H$
  in $G$, a contradiction. We conclude that $G'$ is $H'$-free.

  Finally we argue that adding any missing edge to $G'$ produces a new
  copy of $H'$.  Since $x^*$ is dominating, any missing edge in $G'$ is
  an edge of the form $yz$ where $y,z \in V(G)$. Now $G + yz$
  contains a copy of $H$, since $G$ is $H$-saturated; adding the
  dominating vertex $x^*$ to this copy of $H$ gives a copy of $H'$ in
  $G'+yz$.  
\end{proof}

\section{Sat-sharp graphs}\label{sec:sat-sharp}
For a graph $H$, let
$\satlim(H) = \lim_{n \to \infty} \frac{\sat(H,n)}{n}$, provided that
this limit exists.  Say a graph $H$ is \emph{sat-sharp} if
$\satlim(H) = \frac{\wt(H) - 1}{2}$. Moreover, say that a graph $H$ is
\emph{strongly sat-sharp} if
$\sat(H,n) = \frac{\wt(H) - 1}{2}n + O(1)$. Note that any strongly
sat-sharp graph is also sat-sharp. Also, note that by adopting the
convention that $w(H)=\infty$ when $E(H)=\emptyset$, we can conclude
that any graph with no edges is strongly sat-sharp since
$\sat(H,n)=\infty$ for all $n \geq \sizeof{V(H)}$.

In this section we will show that the classes of sat-sharp graphs and
strongly sat-sharp graphs are each closed under adding isolated and
dominating vertices. To express these results concisely, we write
statements like ``if $H$ is (strongly) sat-sharp, then $H'$ is
(strongly) sat-sharp'' as shorthand for the pair of statements ``if
$H$ is sat-sharp, then $H'$ is sat-sharp; if $H$ is strongly
sat-sharp, then $H'$ is strongly sat-sharp''.

As $K_1$ is strongly sat-sharp, these closure results immediately
imply that all threshold graphs are strongly sat-sharp (as we will
discuss in Section~\ref{sec:threshold}). They also imply that
\emph{any} graph $H$ which can be proven to be (strongly) sat-sharp
gives rise to many (strongly) sat-sharp graphs derived from $H$ by
these operations. In particular, we will prove in
Section~\ref{sec:disjoint-cliques} that a disjoint union of cliques is
strongly sat-sharp, although it is not in general a threshold graph;
this implies that any graph obtained from a disjoint union of cliques
via these operations is also strongly sat-sharp.

\begin{lemma}
  If $H$ is a (strongly) sat-sharp graph, and $H'$ is obtained from $H$ by adding
  isolated vertices, then $H'$ is  (strongly) sat-sharp, and $\satlim(H') = \satlim(H)$.
\end{lemma}
\begin{proof}
  For all $n \geq \sizeof{V(H')}$, a graph $G$ is $H'$-saturated if
  and only if it is $H$-saturated, hence $\sat(H', n) = \sat(H, n)$
  for all sufficiently large $n$.
\end{proof}

To handle the operation of adding a dominating vertex, we prove the
following two lemmas, which taken together show that the class of
(strongly) sat-sharp graphs is closed under the operation of
adding a dominating vertex.
\begin{lemma}\label{lem:sat-dom-smallwt}
  Let $H$ be a $k$-vertex  (strongly) sat-sharp graph, and let $H'$
    be obtained from $H$ by adding a dominating vertex $v^*$.
    If $H$ has no isolated vertices, or if $\wt(H) \leq k-2$,
    then $\wt(H') = 2 + \wt(H)$ and $\satlim(H') = 1 + \satlim(H)$.
   Moreover, $H'$ is  also (strongly) sat-sharp.
\end{lemma}
\begin{lemma}\label{lem:sat-isol-dom-bigwt}
  Let $H$ be a $k$-vertex graph with an isolated vertex $u$,
  and let $H'$ be obtained from $H$ by adding a dominating vertex
  $v^*$.  If $\wt(H) > k-2$, then $H'$ is strongly sat-sharp, with $\wt(H') = k$ and
  $\satlim(H') = \frac{k-1}{2}$.  
\end{lemma}
Note that Lemma~8 does not actually require the graph $H$ to be
sat-sharp, although that is the main case we are concerned with.  In
the case where $H$ has no edges and so $\wt(H) = \infty$, the
hypothesis of Lemma~\ref{lem:sat-isol-dom-bigwt} applies, yielding
$\satlim(K_{1,k}) = \frac{k-1}{2}$; this is an asymptotic version of
the exact result of K\'aszonyi and Tuza~\cite{KT}.

\begin{proof}[Proof of Lemma~\ref{lem:sat-dom-smallwt}]
  Let $\epsilon(H, n) = \sat(H,n) - \satlim(H)n$, so that $\epsilon(H,n) = o(n)$
    when $H$ is sat-sharp and $\epsilon(H,n) = O(1)$ when $H$ is strongly sat-sharp.
  
  By Lemma~\ref{lem:adddom}, we have
   \[\sat(H', n) \leq (n-1) + \sat(H, n-1) = (\satlim(H)+1)n + \epsilon(H, n-1) - 1 .\]   If we can prove that $\wt(H') \geq \wt(H) + 2$, then
   Lemma~\ref{lem:wtlower} will give
   \[\sat(H', n) \geq \frac{\wt(H')-1}{2}n-c'_{H'} = (\satlim(H)+1)n -
     c'_{H'}.\] In particular, this implies that
   $\satlim(H') = \frac{\wt(H')-1}{2}$ and that
   $\sizeof{\epsilon(H', n)} \leq  \sizeof{\epsilon(H,n)} + \sizeof{c'_H} + 1$, so if
   $H$ is (strongly) sat-sharp, then $H'$ is (strongly) sat-sharp.

  An edge $e \in E(H)$ can be viewed (and its weight computed) either
  as an edge of $H$ or as an edge of $H'$.  We will use $\wt(e)$ and
  $\wt'(e)$ to refer to the weight of such an edge computed in $H$ or
  $H'$, respectively. Observe that if $uw \in E(H)$, then when we pass
  to $H'$, we add $v^*$ as a new element of $N(u) \cap N(w)$ and change 
  nothing else about the sets $N(u) \cap N(w)$ or $N(w) - N(u)$. Hence,
  $\wt'(e) = \wt(e) + 2$ for all $e \in E(H')$.

  The only remaining edges of $H'$ are edges of the form $v^*u$ for
  $u \in V(H)$. We claim that all such edges have weight at most
    $2 + \wt(H)$. If $u$ is isolated in $H$, then we have
    \[ \wt'(uv^*) = 2\sizeof{N(u) \cap N(v^*)} +
      \sizeof{N(v^*) - N(u)} = 0 + k = k \geq 2+\wt(H), \] where the
    last inequality follows from the assumption that $\wt(H) \leq k-2$
    (since we assumed that $u$ is isolated and that $H$ either obeys
    this weight inequality or is isolate-free).

    On the other hand, if $u$ is not isolated in $H$, let $ut$ be
    another edge incident to $u$. Observe that
  \begin{align*}
    \wt'(v^*u) &= 2\sizeof{N(u) \cap N(v^*)} + \sizeof{N(v^*) - N(u)} \\
               &= 2d_H(u) + (k - d_H(u)) \\
               &= d_H(u) + k
  \end{align*}
  and that $\wt(ut) \leq (d_H(u)-1) + (k-1)$ for any edge
  $ut \in E(H)$. It follows that
  $\wt'(v^*u) \geq \wt(ut)+2 = \wt'(ut)$. Hence, an edge of minimum
  weight in $H'$ is found among the edges of $H$, and the smallest
  such weight is $\wt(H) + 2$.
\end{proof}
\begin{proof}[Proof of Lemma~\ref{lem:sat-isol-dom-bigwt}]
  We again write $\wt(e)$ to refer to the weight of an edge $e$
  computed in $H$ and $\wt'(e)$ to refer to the weight of an edge $e$
  when computed in $H'$.

  As previously discussed, we have $\wt'(e) = \wt(e) + 2$ for every edge $e \in E(H)$.
  On the other hand, considering the isolated vertex $u$, we see that $\wt(uv^*) = k$,
  as $\sizeof{N(u) \cap N(v^*)} = 0$ and $\sizeof{N(v^*) - N(u)} = k$.
  
  If $\wt(H)+2 > k$, then this implies $\wt(H') = k$, with the only
  edges of minimum weight being those edges joining $v^*$ with an
  isolated vertex of $H$. This establishes the first claim of the lemma.

  Lemma~\ref{lem:wtlower} now gives the lower bound
 \[\sat(H',n) \geq \frac{k-1}{2}n - c'_{H'}.\] We establish a matching upper bound by constructing an
  $H'$-saturated graph on $n$ vertices, for any $n \geq \sizeof{V(H')}$.

  Let any $n \geq \sizeof{V(H')}$ be given, and write $n = qk + r$,
  where $0 \leq r < k$. Let $G$ be the $n$-vertex graph consisting of $q$ disjoint copies of $K_k$
  and a single copy of $K_r$. Clearly
  \[ \sizeof{E(G)} =\frac{n-r}{k}{k \choose 2} + {r \choose 2} \leq
    \frac{k-1}{2}n.\] So if we can argue that $G$ is
  $H'$-saturated, then we will have $\satlim(H') = \frac{k-1}{2}$, and
  we will have that $H'$ is strongly sat-sharp.

  It is clear that $G$ is $H'$-free, since $H'$ is connected and has
  $k+1$ vertices, while every component of $G$ has at most $k$
  vertices. We claim that adding any edge to $G$ produces a subgraph
  isomorphic to $H$. Let $xy$ be a missing edge in $G$; we may assume
  that $y$ lies in a copy of $K_k$. Let $Q$ be the set of vertices of
  the copy of $K_k$ containing $y$.

  Now observe that we can embed $H'$ into $G+xy$ by any injection $\phi : V(H') \to V(G)$
  that satisfies:
  \begin{itemize}
  \item $\phi(u) = x$, and    
  \item $\phi(v^*) = y$,
  \item $\phi(V(H) - \{u,v^*\}) = Q - y$,
  \end{itemize}
  and with $k-1$ vertices in $Q-y$, there is enough room to complete
  the last part of the embedding.  The key point is that there is no
  edge, in $H'$, from $u$ to any vertex of $H'$ except for $v^*$, and
  all vertices of $H'$ except for $u$ are being embedded into a clique
  of $G$, so any edges they require are present. Thus, $G$ is
  $H'$-saturated, which completes the proof.
\end{proof}
\section{Threshold Graphs}\label{sec:threshold}
A natural class of strongly sat-sharp graphs is the class of \emph{threshold graphs}. A
simple graph $G$ with vertex set $\{v_1, \ldots, v_n\}$ is a
\emph{threshold graph} if there exist weights
$x_1, \ldots, x_n \in \reals$ such that, for all $i \neq j$, we have
$v_iv_j \in E(G)$ if and only if $x_i + x_j \geq 0$. Threshold graphs
were first introduced by Chv\'atal and Hammer~\cite{CH73, CH77},
albeit with a slightly different definition than the one we give here.

Threshold graphs admit many equivalent characterizations. For our
purposes, the following characterization is the most useful one.
\begin{theorem}[Chv\' atal--Hammer~\cite{CH77}; see also~\cite{MP}]
  For a simple graph $G$, the following are equivalent:
  \begin{enumerate}
  \item $G$ is a threshold graph;
  \item $G$ can be obtained from $K_1$ by iteratively adding a new vertex
    which is either an isolated vertex, or dominates all previous vertices;
  \end{enumerate}
\end{theorem}
In fact, \cite{MP} gives several other equivalent characterizations of
threshold graphs, but this is the one we will be interested in.
The results of Section~\ref{sec:sat-sharp}, together with this
characterization, immediately imply that all threshold graphs are
 strongly sat-sharp. Furthermore, when a construction sequence for a threshold graph $G$
is given, one can use the lemmas from Section~\ref{sec:sat-sharp} to easily
compute $\wt(G)$ by iteratively computing the weight of each intermediate
subgraph, keeping track of the previous subgraph's weight and whether
or not it had an isolated vertex.

As discussed in the introduction, several graphs whose saturation
numbers were determined in previous work fall into the class of
 strongly sat-sharp graphs.  In particular, complete graphs~\cite{EHM},
stars~\cite{KT}, generalized books~\cite{CFG08}, stars plus an edge
\cite{FFGJ-trees}, and ``nearly complete'' graphs~\cite{FG13} of the form $K_t - H$
for $H \in \{K_{1,3},K_4-K_{1,2},K_4-K_2\}$ are all threshold graphs.
Thus, all of these graphs are strongly sat-sharp, and their saturation
number is determined (up to a constant number of edges) by the
results of Section~\ref{sec:sat-sharp}.

As a non-example, we note that among the ``nearly complete'' graphs of
\cite{FG13}, the graph $K_t - 2K_2$ is not a threshold graph, and in
fact \cite{FG13} prove that
$\sat(K_t - 2K_2) = (t-\frac{5}{2})n + O(1)$, whereas
Lemma~\ref{lem:wtlower} only gives the lower bound
$\sat(K_t-2K_2, n) \geq (t - \frac{7}{2})n$.

K\' aszonyi and Tuza~\cite{KT} observed the ``irregularity'' that if
$H$ is the graph obtained from $K_4$ by adding a pendant edge, then
$\sat(H, n) \leq \frac{3}{2}n$ while $\sat(K_4, n) = 2n-3$, so that
$\sat(H,n) < \sat(K_4, n)$ for sufficiently large $n$ even though
$K_4 \subset H$. Both $K_4$ and the graph $H$ are threshold graphs; in
terms of our weight function, the irregularity can be seen to arise
from the fact that all edges of $K_4$ have weight $5$ while the
pendant edge of $H$ has weight $4$.

\section{$H$-saturated construction when $H$ is the disjoint union of cliques}\label{sec:disjoint-cliques}

Faudree, Ferrara, Gould, and Jacobson~\cite{FFGJ09} determined the
saturation numbers of generalized friendship graphs $F_{t,p,\ell}$,
consisting of $t$ copies of $K_p$ which all intersect in a common
$K_\ell$ but are otherwise pairwise disjoint. When $\ell=0$, this includes
the case of $tK_p$, consisting of $t$ disjoint copies of $K_p$. They
also determined the saturation numbers of two disjoint cliques,
$K_p \cup K_q$, when $p \neq q$, but left determining the saturation
number of three or more disjoint cliques with general orders as an
open problem. Here, we give a proof that all of these graphs are
 strongly sat-sharp, and determine their saturation numbers up to an additive
constant for all sufficiently large $n$.

\begin{proposition}\label{prop:disjoint-cliques}
  Let $2 \leq p_1 \leq \cdots \leq p_m$ be positive integers. The
  graph $H=K_{p_1} \cup \cdots \cup K_{p_m}$ is  strongly sat-sharp. In
  particular, \[(p_1-2)n-c'_H \leq \sat(H,n) \leq (p_1-2)n + c_H\] for
  some constants $c_H,c'_H$ depending only on $H$ and for all
  $n \geq \sum_{i=1}^m p_i$.
\end{proposition}

\begin{proof}
  First, note that $\wt(H)=2(p_1-2)+1$. So by
  Lemma~\ref{lem:wtlower}, \[\sat(H,n) \geq (p_1-2)n-c'_H\] for some
  constant $c'_H$.

On the other hand, let $G$ be the graph on $n$ vertices defined as the
join, $G=K_{p_1-2} \vee G'$ where $G'=K_{t} \cup I$, the disjoint
union of a clique on $t= 1+ \sum_{i=2}^m p_i $ vertices and a set $I$
with $n-t-p_1+2$ isolated vertices. Figure~\ref{fig:disclique-sat} shows
  the graph $G$ that is constructed for $H = H_4 \cup K_5 \cup K_6$.
\begin{figure}
  \centering
  \begin{tikzpicture}
    \begin{scope}[yshift=.5cm]
      \foreach \i in {1,...,12}
      {
        \apoint{} (v\i) at (30*\i : 1cm) {};
        \foreach \j in {1,...,\i}
        {
          \draw (v\i) -- (v\j);
        }
      }      
    \end{scope}
    \begin{scope}[yshift=-1.25cm, yscale=.15]
      \apoint{} at (-1.5cm, 0cm) {};
      \apoint{} at (-1cm, 0cm) {};
      \apoint{} at (-.5cm, 0cm) {};      
      \node at (0cm, 0cm) {$\cdots$};
      \apoint{} at (.5cm, 0cm) {};            
      \apoint{} at (1cm, 0cm) {};
      \apoint{} at (1.5cm, 0cm) {};
      \draw (0cm, 0cm) circle (1.75cm);
      \node[anchor=north] at (0cm, -1.75cm) {$I$};          
    \end{scope}
    \node[anchor=north west] at (-2cm, 2cm) {$G'$};
    \draw (-2cm, 2cm) rectangle (2cm, -2cm) {};
    \apoint{} (u) at (-4cm, 1cm) {};
    \apoint{} (v) at (-4cm, -1cm) {};
    \draw (u) -- (v);
    \foreach \i in {0,...,10}
    {
      \node[coordinate] (z\i) at (-2cm, 2cm-.4*\i cm) {};
      \draw (u) -- (z\i);
      \draw (v) -- (z\i);
    }
  \end{tikzpicture}
  \caption{Construction of the saturated graph $G$ for $H = K_4 \cup K_5 \cup K_6$.}
  \label{fig:disclique-sat}
\end{figure}
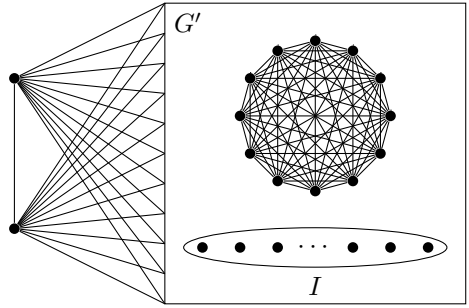

We claim that $G$ is $H$-free and $H$-saturated. To see that $G$ is
$H$-free, consider its maximal cliques. Let $Q$ denote the subgraph of
$G$ induced by the vertices of the $K_{p_1-2}$ and the $K_t$. Then $Q$
is a maximal clique with $p_1-2+t$ vertices. All other maximal cliques
of $G$ are formed from the $K_{p_1-2}$ and one vertex from
$I$. Therefore, if we were to find a copy of $H$ in $G$, then each of
the disjoint cliques of $H$ must be found in $Q$. But $Q$ only has
$|V(H)|-1$ vertices so this cannot happen.

To see that $G$ is $H$-saturated, consider the graph $G+xy$ for
  some $xy \notin E(G)$. Without loss of generality, either
$x,y \in I$ or $x \in I$ and $y \in K_t$. In either case, the vertices
of $K_{p_1-2} \cup \{x,y\}$ form a $p_1$-clique in $G+xy$, while at
least $t-1=p_2 + \cdots +p_m$ vertices of the $K_t$ remain disjoint
from this clique and can be used to embed the remaining cliques of
$H$. So $G$ is $H$-saturated.

Since $G$ has $(p_1-2)\left(n+1-\sum_{i=1}^m p_i\right)+{p_1+\cdots+p_m-1 \choose 2}$ edges, it follows that \[\sat(H,n) \leq (p_1-2)n + c_H\] for some constant $c_H$. Therefore, $H$ is  strongly sat-sharp.
\end{proof}

An immediate corollary to this proposition and the results of Section~\ref{sec:sat-sharp} is the following result.

\begin{corollary}
  Let $\ell$ and $2 \leq p_1 \leq \cdots \leq p_m$ be positive
  integers. Let $H' = K_{p_1} \cup \cdots \cup K_{p_m}$, and let
  $H=K_{\ell} \vee H'$. Then $H$ is sat-sharp. In
  particular, \[(p_1+\ell-2)n-c'_H \leq \sat(H,n) \leq (p_1+\ell-2)n + c_H\]
  for some constants $c_H,c'_H$ depending only on $H$ and for all
  $n \geq \ell+ \sum_{i=1}^m p_i$.
\end{corollary}

Note that this class of graphs includes all generalized friendship
graphs $F_{t,p,\ell}$ for $p \geq l+2$. Since $F_{t,p,\ell}$ for
$p=\ell+1$ is a threshold graph, we already know from the discussion
in Section~\ref{sec:threshold} that it is  strongly sat-sharp.

While a disjoint union of cliques is not, in general, a threshold
graph, each of its components is a threshold
graph. Proposition~\ref{prop:disjoint-cliques} therefore suggests that
perhaps a disjoint union of threshold graphs is always strongly sat-sharp.
More boldly, the following conjecture appears to be plausible:
\begin{conjecture}\label{coj:disjoint-union}
  If $H_1$ and $H_2$ are (strongly) sat-sharp graphs, then their
  disjoint union $H_1 + H_2$ is (strongly) sat-sharp. That is, the
  class of (strongly) sat-sharp graphs is closed under taking disjoint
  unions.
\end{conjecture}
Conjecture~\ref{coj:disjoint-union}, together with the other closure
properties from Section~\ref{sec:sat-sharp}, would immediately imply
Proposition~\ref{prop:disjoint-cliques}. We have found ad-hoc
constructions for some small disjoint unions of particular threshold
graphs which suggest that Conjecture~\ref{coj:disjoint-union} might
hold, but it has been difficult to extract a general construction.

\section*{Acknowledgements}

We would like to thank Ron Gould for the interesting talk on
saturation numbers that he gave at the Atlanta Lecture Series in 2018
which stimulated work on this paper.

\bibliographystyle{plain} \bibliography{bibthresh}
\end{document}